\documentclass[11pt,reqno]{amsart}
\usepackage{graphicx}
\usepackage{amsmath,amsthm,amssymb,enumerate, esint}
\usepackage{color}
\usepackage{wrapfig}
\usepackage{caption}
\usepackage{geometry}

\usepackage{xstring}

\newcommand{\BR}[1]{%
    \IfEqCase{#1}{%
        {6}{\color{black}} 
    }[\PackageError{BC}{Undefined option to tree: #1}{}]%
}
\newcommand{\ER}{\color{black}}

\newtheorem{theorem}{Theorem}
\newtheorem{proposition}{Proposition}[section]
\newtheorem{lemma}{Lemma}[section]

\providecommand{\uoset}[3]{\underset{\phantom{#1}}{\overset{#2}{#3}}}

\newcommand{\R}{\mathbb{R}}
\providecommand{\oh}{\frac{1}{2}}
\providecommand{\ud}{\, \mathrm{d}}

\providecommand{\dx}{\ud x}

\providecommand{\dt}{\ud \theta}
\providecommand{\dr}{\ud r}

\providecommand{\bv}{\bar v}
\providecommand{\bw}{\bar w}

\providecommand{\W}{W_{\textrm{rel}}}
\providecommand{\h}{h_0}

\DeclareMathOperator*{\argmin}{arg\,min}

\begin{document}


\title{Wrinkling of a thin circular sheet bonded to a spherical substrate}

\author{Peter Bella}
\address{Institute of Mathematics, Leipzig University,\\
Augustusplatz 10, 04109 Leipzig, Germany}\email{bella@math.uni-leipzig.de}
\author{Robert V. Kohn}
\address{Courant Institute of Mathematical Sciences, New York University, 251 Mercer Street, New York, NY 10012, USA}\email{kohn@cims.nyu.edu}

\begin{abstract}
We consider a disk-shaped thin elastic sheet bonded to a compliant sphere.
(Our sheet can slip along the sphere; the bonding controls only its
normal displacement.) If the bonding is stiff (but not too stiff), the geometry of the sphere
makes the sheet wrinkle to avoid azimuthal compression. The total energy of this system is the elastic energy of the sheet plus a (Winkler-type) substrate energy. Treating the~thickness of the sheet $h$ as a small parameter, we determine the leading-order behavior of the energy as $h$ tends to zero, and give (almost matching) upper and lower bounds for the next-order correction. Our analysis of the leading-order behavior determines the~macroscopic deformation of the sheet; in particular it determines the extent of the wrinkled region, and predicts the (nontrivial) radial strain of the sheet. 
The~leading-order behavior also provides insight about the length 
scale of the wrinkling, showing that it must be approximately
independent of the~distance~$r$ from the center of the sheet (so that the 
number of wrinkles must increase with $r$). Our results on the next-order 
correction provide insight about how the wrinkling pattern should vary with $r$. 
Roughly speaking, they suggest that the length scale of wrinkling should
{\it not} be exactly constant -- rather, it should vary a bit, so that  
the number of wrinkles at radius $r$ can be approximately piecewise constant
in its dependence on $r$, taking values that are integer 
multiples of $h^{-a}$ with $a \approx 1/2$.
\end{abstract}

\maketitle



\section{Introduction}
We model the wrinkling of a disk-shaped elastic sheet bonded to a compliant sphere, 
as shown schematically in Figure~\ref{fig1}. The source of the wrinkling is easy to 
understand: if we assume for a moment that the sheet is inextensible in the radial 
direction and that the center of the disk is attached to the north pole, then each 
circle $|x|=r$ is approximately mapped to the circle $S_r$ on the sphere at distance 
$r$ from the north pole. Since the arclength of $S_r$ is less than $2\pi r$, circles 
$|x|=r$ must wrinkle to avoid (large) compression. Roughly: the typical slope of the 
wrinkling is determined by the geometry of the sphere (the contrast between $2\pi r$ and 
$|S_r|$), while the~wavelength is determined by competition between the bending energy 
(which prefers coarse, large-amplitude wrinkling) and the substrate energy (which prefers 
small deformations, hence fine, low-amplitude wrinkling).
 
The preceding account is oversimplified. Our sheet is not inextensible, and we permit
it to slip along the sphere. By stretching a bit in the radial direction the
sheet can reduce the energetic cost of wrinkling, since the circle
$|x|=r$ is then approximately mapped to a circle on the sphere slightly
longer than $S_r$. As we'll explain in due course, the macroscopic deformation
of our sheet is determined by \BR6 the \ER competition between membrane effects (which prefer
less stretching) and the~energetic cost of wrinkling (which prefers more stretching).

 \begin{wrapfigure}[24]{l}{6cm}
  \includegraphics[scale=0.4]{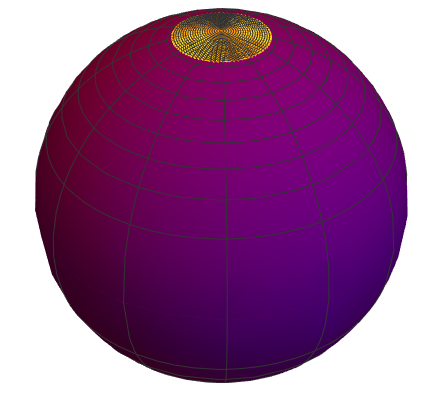}
   \includegraphics[scale=0.4]{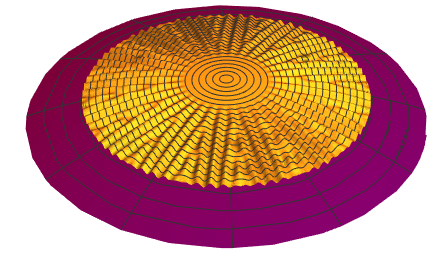}
 \vspace{-14pt}
  \caption{Circular sheet on a ball}\label{fig1}
 \end{wrapfigure}

The behavior of thin elastic sheets experiencing compression due to geometric 
effects has recently received a lot of attention. Without attempting a comprehensive review, let us mention studies concerning a sheet on a deformable sphere~\cite{HohlfeldDavidovitch, Paulsen++,DavidovitchGrasonSun}; indentation of a pressurized ball~\cite{VellaEbrahimi+2015}; indentation of a floating sheet~\cite{VellaHuang+2015, Paulsen++}; 
wrinkling of a stamped plate~\cite{hure2012stamping};
and crystalline sheets on curved surfaces~\cite{GrasonDavidovitch,meng2014elastic}. 
Among these references the paper~\cite{HohlfeldDavidovitch} deserves
special note, since (as we explain in Section~\ref{sec:model}) our model is
particularly close to the one considered there. 


It is well known that with increasing compression a thin elastic sheet undergoes an 
instability (like Euler buckling), the onset of which is well-understood 
using linear analysis (this is \BR6 the\ER~so~called ``near-threshold'' (NT) regime). 
As the compression increases one enters a~different, ``far-from-threshold'' (FT) regime (see e.g.~\cite{benny1}), in which predictions from the linear 
theory cease to be valid. In contrast with \BR6 the \ER NT regime, in the FT regime 
the sheet (almost) completely releases the~compressive stresses by 
deforming out-of-plane (e.g. by wrinkling). The~wrinkling wavelength 
is then set by \BR6 a \ER competition between the bending resistance (which prefers 
long wavelengths) and mechanisms favoring short wavelengths (e.g. tension, 
curvature along the~wrinkles, and adhesion to a substrate). The natural 
goals in the FT regime are to predict the wavelength of wrinkles (by 
deriving a so called ``local $\lambda$-law''\BR6 ~\cite{cerdamaha1,Paulsen++}) \ER and/or 
to predict the~macroscopic deformation of the sheet. These goals are the primary focus of many of the papers cited above  
~\cite{benny1,GrasonDavidovitch,HohlfeldDavidovitch,Paulsen++,VellaEbrahimi+2015,VellaHuang+2015}.

While our goal in the present paper is very similar, there is an unexpected twist 
compared to the aforementioned work. There the energy consists of a {\it dominant} 
part which decides the~macroscopic deformation, and a {\it subdominant} part which controls the scale of the wrinkling. Put differently: 
in the limit of vanishing thickness the wrinkling does not cost any energy (since 
the energetic contribution from wrinkling is subdominant), and the macroscopic deformation of the sheet can be obtained via tension-field theory (in mathematical language: by 
minimizing a~relaxed functional). In contrast, in the problem we consider the 
cost of wrinkling is comparable to other terms in the dominant energy; as a result 
one cannot use tension-field theory or solve a~relaxed problem to predict the macroscopic 
deformation of the sheet. Instead, one must minimize an~\emph{effective} 
functional, in which the elastic energy of radial tension competes with 
the (substrate + bending) energy of circumferential wrinkling. 
Since the energetic cost of wrinkling contributes to the leading order term 
in the energy, 
minimization of the effective energy determines (at least approximately) the length
scale of wrinkling at radius $r$.  
Our problem has this character because we consider a stiff elastic substrate, 
quite different in character from the~liquid substrate considered in~\cite{VellaHuang+2015}
and stiffer than the relatively compliant Winkler foundations considered 
in~\cite{HohlfeldDavidovitch,Paulsen++,VellaEbrahimi+2015,GrasonDavidovitch} 
(see Section~\ref{sec:model} for more about this).

As already mentioned, the minimum of our effective functional determines the 
macroscopic deformation and the limiting energy as the thickness $h \rightarrow 0$. 
But more: it gives a {\it lower bound} for the energy $E_h$ when $h$ is positive. (For the precise definition of $E_h$ see~\eqref{energy} below.) 
The amount by which $E_h$ exceeds the minimum of the effective functional is informative; 
therefore the estimation of this {\it excess energy} is a major focus of our work. Our main
mathematical result, Theorem~\ref{thm}, provides upper and lower bounds on the excess energy, 
showing (very roughly speaking) that it is approximately linear in $h$. 

Our estimate of the excess energy has implications for the fine-scale structure of the
wrinkling. This is because $E_h$ includes the cost of changing the wrinkling pattern as a 
function of the~distance $r$ from the center of the sheet, whereas the effective 
functional ignores this cost. Indeed, the~effective functional estimates the 
length scale of wrinkling by balancing the azimuthal bending against the substrate term; 
this leads to the conclusion (found also in~\cite{HohlfeldDavidovitch}) that the~scale
of the~wrinkling should be approximately proportional to $h$ (independent of $r$). 
It follows that the number of wrinkles at $|x|=r$ should increase approximately linearly
with $r$. Our analysis of the upper bound on the excess energy shows that the length scale
of wrinkling should {\it not} be exactly constant -- rather, it 
should vary a bit, so that the number of wrinkles at radius $r$ can be 
approximately piecewise constant in its dependence on $r$, taking values that are integer 
multiples of $h^{-a}$ with $a \approx 1/2$.

The picture that emerges has a lot of symmetry: the macroscopic deformation (determined by
minimizing the effective energy) involves radial tension, and the number of wrinkles at radius
$r$ is an approximately (but not exactly) linear function of $r$. This symmetry is 
a {\it conclusion}, not a~hypothesis, of our analysis. It is of course crucial that our sheet is
disk-shaped, and that the~sphere is a body of revolution around the axis determined by the center
of the sheet.

It is not a new idea that energy scaling laws can be used to identify the wrinkled region 
and to explain the local length scale of wrinkling. The best-understood examples are 
problems where the geometry is simple (typically flat) and the direction of wrinkling 
is fixed by some source of uniaxial tension (e.g. a stretched annular 
sheet~\cite{BellaKohnCPAM,BellaARMA} or a hanging drape~\cite{BellaKohnDrapes}).
Problems involving biaxial compression are less well-understood, though there has 
been progress in special cases (e.g. the shape of a blister in a~compressed thin 
film~\cite{blisters-linear, jin-sternberg,blisters2}, and the herringbone pattern 
seen in a~compressed thin film bonded to compliant substrate~\cite{kohn-nguyen}). 
The problem considered here involves compression, but its geometry is rather 
controlled due to the presence of the substrate and the use of von K\'arm\'an theory. 

The paper is organized as follows. Section~\ref{sec:model} presents our model, states our main mathematical results (Theorem~\ref{thm}), 
and provides further discussion about their implications. Sections~\ref{sec:lb} and~\ref{sec:prop} 
prove the lower-bound half of Theorem~\ref{thm}. The argument relies on certain properties of 
the~minimizers of some one-dimensional calculus of variations problems closely related to our
effective functional. Section~\ref{sec:lb} states the required properties in 
Proposition~\ref{prop} then uses them to prove the bound, while Section~\ref{sec:prop}
provides the proof of Proposition~\ref{prop}. Finally, Section~\ref{sec:ub} proves the upper-bound
half of Theorem~\ref{thm}. This is done by identifying an explicit wrinkling pattern (varying 
appropriately with $r$) with relatively small excess energy. While the pattern given there is 
not the energy minimizer (we do not solve an Euler-Lagrange equation), it provides an indication 
about how a wrinkling pattern should look in order to achieve the minimum energy scaling law. 

\section{The model and the main results}\label{sec:model}

We consider a thin elastic sheet of circular shape with thickness $h > 0$ and radius $r_0 > 0$, 
which sits on an elastic ball of radius $R \gg r_0$. The energy of the system has three terms: 
the {\it membrane} energy of the sheet, which measures deviation from the deformation being an 
isometry; the {\it bending} energy of the sheet, which penalizes curvature; and 
a {\it substrate} energy, which prefers the sheet to be sphere-shaped. 
For the membrane term we use a F\"oppl-von K\'arm\'an model (taking Poisson's ratio
equal to zero for simplicity); for the substrate term we use a Winkler foundation. Focusing on
the energy per unit thickness and normalizing by the Young's modulus of the sheet, our elastic energy
functional is 
\begin{equation}\label{energy}
E_h(u,\xi) := \int_\Omega 
\biggl|e(u) + \oh \nabla \xi \otimes \nabla \xi\biggr|^2 + 
h^2 |\nabla\nabla \xi|^2 \dx + 
\alpha_s h^{-2} \int_{\Omega} \biggl|\xi + \frac{|x|^2}{2R} \biggr|^2 \dx.
\end{equation}
Here $\Omega$ denotes the disk of radius $r_0$, centered at the origin; $u$ and $\xi$ are
the in-plane and out-of-plane displacements of the sheet, and $e(u) = (\nabla u + (\nabla u)^T)/2$ 
denotes the linear strain associated with $u$. The nondimensional constant $\alpha_s$ 
(which we assume is strictly positive) determines the~relative stiffness of the substrate 
compared to that of the film. Since the substrate term involves only $\xi$ and not $u$, our model
requires that the sheet conform to the sphere, but permits it to slide along the sphere. (The 
standard F\"oppl-von K\'arm\'an bending term would be $\frac{1}{12} |\nabla \nabla \xi|^2$; we have
dropped the factor $1/12$ for notational simplicity. This simplification does not change the problem significantly, though it affects the precise form of the effective functional.)

Our energy~(\ref{energy}) is very similar to the one considered by Hohlfeld and Davidovitch
in~\cite{HohlfeldDavidovitch}. There are, however, two significant differences: (1) their 
energy has an additional term, representing surface tension, which induces a state of (small) radial
tension; (2) they focus on the~limit of ``asymptotic isometry'', which is achieved when 
both the surface tension and our $\alpha_s$ tend to $0$. Our situation is different, because
we take $\alpha_s$ to be nonzero (and fixed) as $h \rightarrow 0$. As a~result, our sheet does not
achieve asymptotic isometry (despite the absence of surface tension); rather, it is in a 
state of radial tension.
\BR6
To make the comparison with~\cite{HohlfeldDavidovitch} more explicit: their substrate term is 
$\frac{K}{E_f h} \int_{\Omega} \big|\xi + \frac{|x|^2}{2R} \big|^2 \dx$ -- the same as ours except that coefficient is $\frac{K}{E_f h}$ instead of
$\alpha_s h^{-2}$. (Here $E_f$ is the elastic modulus of the sheet and
$K$ is a constant determinicng the stiffness of their substrate.) Thus,
their analysis (sending $h \rightarrow 0$ while holding $K$ fixed)
corresponds, in our notation, to considering $\alpha_s = \frac{K}{E_f} h$.
This is {\it not} the regime we consider; rather, our $\alpha_s$ is
fixed and positive as $h \rightarrow 0$.
\ER

We choose to work in dimensional variables: since $u$, $\xi$, and $h$ have the dimensions of 
length, our energy~(\ref{energy}) has dimension $\mbox{length}^2$. However the problem would not be
significantly different if we worked in nondimensional variables, replacing $x$ by $x/R$, etc. 
The nondimensional version of the energy looks the same as~(\ref{energy}) except that the domain is
a ball of radius $r_0/R$, the factor in front of $|\nabla \nabla \xi|^2$ is $(h/R)^2$, and 
the substrate term is $\alpha_s (h/R)^{-2} \int \bigl| \xi + \frac{|x|^2}{2} \bigr|^2 \dx$. 

This model can be criticized on the ground that the Winkler substrate term is somewhat idealized. 
A more realistic model of a film bonded to a uniform elastic ball would replace our Winkler term with
a nonlocal expression involving the $H^{1/2}$ norm of the surface displacement (see for 
example~\cite{kohn-nguyen} or Section 3 of~\cite{audoly2008buckling}). This suggests
replacing our Winkler term by 
\begin{equation} \label{nonlocal-substrate-term}
\alpha_s h^{-1} \biggl\|\xi(\cdot) + \frac{|\cdot|^2}{2R} \biggr\|_{H^{1/2}(\Omega)}^2
\end{equation}
(which scales once again like $\mbox{length}^2$). In fact our analysis of the lower 
bound would also work for this nonlocal substrate term. However our analysis of the upper
bound requires estimating the energy of an explicitly-given deformation, which would be much
more difficult for a nonlocal model. Moreover the use of~(\ref{nonlocal-substrate-term}) would not
eliminate another key idealization, namely that the~sheet is free to slip along the substrate.
In addition, our goal is to consider a thought-experiment not a physical experiment: how 
does geometry induce wrinkling, when a thin elastic sheet is required to conform to a sphere? 
With these considerations in mind, we take the view that \BR6 the\ER~Winkler model used in~(\ref{energy})
is appropriate for our purposes. 

\begin{wrapfigure}[18]{l}{5cm}
\includegraphics[width=5cm]{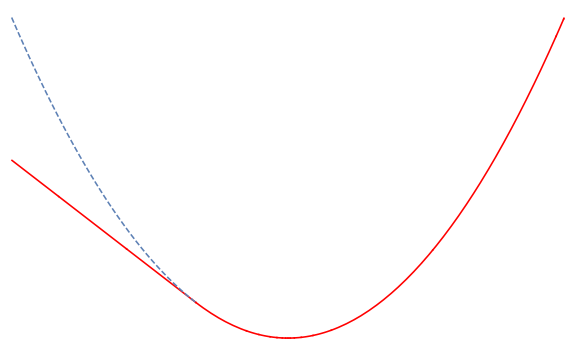}
\caption{The graph of $\W(\eta)$, superimposed on the parabola
$\eta^2$. In the region corresponding to wrinkling $\W(\eta)$ is linear;
for values of $\eta$ corresponding to an unwrinkled state, $\W(\eta)=\eta^2$.}
\label{fig:wrel}
\end{wrapfigure}

In fact, Winkler-type substrate terms have been used to model many experiments. 
In~\cite{HohlfeldDavidovitch, Paulsen++, VellaEbrahimi+2015, GrasonDavidovitch} 
(see also \cite{VellaHuang+2015}) such a term arises from the gravitational potential of 
the fluid below the sheet. There is, however, an important difference: in that work the prefactor scales like $h^{-1}$, whereas in~(\ref{energy}) it scales like $h^{-2}$. 
By considering one-dimensional examples, one sees that in our setting the cost of 
wrinkling (determined by optimizing the length scale, based on competition between the 
bending and substrate terms) is $O(1)$ (half way between $h^2$ and $h^{-2}$); this is why it
contributes to the leading-order energy. 
In~\cite{HohlfeldDavidovitch, Paulsen++, VellaEbrahimi+2015, GrasonDavidovitch}, by contrast, 
the different scaling makes the cost of wrinkling $o(1)$ (so it is subdominant).

We turn now to a discussion of our results. They involve (1) identification of an \emph{effective
functional} $F_0$, whose minimum determines the radial strain in the sheet, the approximate
length scale of wrinkling, and the limiting behavior of the minimum elastic energy ($\min E_h$)
as $h \rightarrow 0$; and (2) upper and lower bounds for the \emph{excess energy}, defined as the 
difference between $\min E_h$ and $\min F_0$.

To describe the effective functional we must first discuss the energetic cost of wrinkling.
As explained in the~introduction, wrinkling is a way for a circle $|x|=r$ to fit into 
less space. We shall show that the elastic cost of such a~circle ``shrinking by amount $-\eta$'' is 
well-approximated by $\W(\eta)$, where
\begin{equation} \label{wrel}
\W (\eta) :=
\begin{cases} \eta^2  & \textrm{if }\eta \ge -2 \alpha_s^{1/2}, \\ 
-4 \alpha_s^{\oh} (\alpha_s^{\oh} + \eta) & \textrm{if } \eta \le -2 \alpha_s^{1/2}. 
\end{cases}
\end{equation}
More precisely, $\W(\eta)$ estimates (with very small error) the cost of fitting a 
circle of length $1-\eta$ into a unit interval. If $\eta > 0$ (or slightly negative), it is 
optimal to deform the~curve uniformly in plane; in this regime the cost is entirely
membrane energy and $\W(\eta) = \eta^2$. However if $\eta$ is negative enough, it is 
better to deform the curve out-of-plane (to \emph{waste arclength by wrinkling} so to speak); in 
this regime the cost is obtained by optimizing the length scale of the wrinkling, and the resulting
formula is linear in $\eta$. The function $\W(\eta)$ is continuously differentiable, though its 
second derivative is discontinuous at the boundary between the two regimes (see Figure~\ref{fig:wrel}).

As noted early in the Introduction, we expect the sheet to be in a state of radial tension, since stretching a bit in the radial direction reduces the energetic cost of wrinkling. The effective
functional $F_0$ captures this effect, by keeping (only) the energy due to radial stretching and 
our estimate for the cost of wrinkling:
\begin{equation}\label{fnull}
F_0(\bv) := \int_0^{r_0} 
\biggl[ \biggl(\bv' + \frac{r^2}{2R^2} \biggr)^2 + \W(\bv / r)\biggr] r \dr.
\end{equation}
Here $\bv$ is a real-valued function of one variable ($r=|x|$), constrained by the boundary condition
$\bv(0) =0$. We shall show that minimization of this one-dimensional variational problem 
provides a~good estimate for the radial deformation of the sheet. In addition, it provides a lot of
information about the~length scale of the wrinkling at radius $r$ (via the analysis that led to
$\W$).

Assuming that the sheet actually wrinkles (i.e. assuming that the excess length of the outer 
circles is large enough to induce wrinkling rather than compression), the cost of changing
the length scale of wrinkling from circle to circle does not enter the leading-order energy; 
rather, it contributes to the principal correction (the excess energy). Since we are 
interested in how the~wrinkling pattern changes with $r$, it is crucial to understand the
principal correction. In fact it seems to be of order $h$ (more precisely: we have a lower 
bound that is linear in $h$, and an upper bound that is almost linear in $h$). 

Our main mathematical result is the following characterization of the leading-order energy 
and the principal correction:

\begin{theorem}\label{thm}
Assume $\alpha_s \in (0,2^{-8} (\frac{r_0}{R})^4)$, and let $\epsilon$ be the 
\emph{excess energy}, defined by
\begin{equation}\nonumber
\epsilon := \inf_{u,\xi} E_h(u,\xi) - \min_{\bv(0)=0} F_0(\bv).
\end{equation}
Then there exist constants $\h > 0$, $c_0 > 0$, and $c_1 < \infty$, all depending on 
$\alpha_s, r_0$, and $R$, such that for any $h \in (0,\h)$ the excess 
energy $\epsilon$ satisfies
\begin{equation}\nonumber
c_0 h \le \epsilon \le \kappa\left(\frac 1h\right) h,
\end{equation}
where the correction 
$\kappa(t) := \exp\Big(c_1\left(\log t \right)^{1/2} \log \big(\log t\big)\Big)$ 
grows slower than $t^{\alpha}$ for any $\alpha > 0$.
\end{theorem}

Since the minimum of the effective functional gives the leading-order elastic energy as 
$h \rightarrow 0$, it is important to note that its minimizer can be made entirely explicit. 
In fact, assuming as in Theorem~\ref{thm} that $\alpha_s < 2^{-8} (\frac{r_0}{R})^4$ and using
the Euler-Lagrange equation for $F_0$, it is straightforward to compute that 
\begin{equation}\begin{aligned}\label{vone}
\argmin_{\bv(0)=0} F_0 = \bv_0(r) &:= \begin{cases} 
\frac{-3}{16} \frac{r^3}{R^2} + \left(2\alpha_s^{\frac12} (r_0/r_w - 1) + 
\frac{1}{16} \frac{r_w^2}{R^2} \right)r & \quad r \in (0,r_w)\\
-2\alpha_s^{\frac12}r - \frac16 \left( \frac{r^3 - r_w^3}{R^2} \right) + 
2 \alpha_s^{\frac12}r_0 \log \left(\frac{r}{r_w}\right) & \quad r \in (r_w,r_0),
\end{cases}
\end{aligned}\end{equation}
with $r_w := (16 \alpha_s^{\frac12} r_0 R^2)^{\frac13}$. The first interval $(0,r_w)$ 
in~\eqref{vone} corresponds to the unwrinkled region, while the second interval 
$(r_w,r_0)$ corresponds to the wrinkled region. The condition 
$\alpha_s < 2^{-8} (\frac{r_0}{R})^4$ is equivalent to $r_w < r_0$, i.e. it 
ensures that the wrinkled region is nontrivial. 

To explain the relationship between our effective functional $F_0$ and the elastic energy $E_h$, 
we note that $F_0$ is closely related to the functional $F_h$, defined for 
$(\bv,\bw) : (0,r_0) \to \R$ by
\begin{equation}\label{fh}
F_h(\bv,\bw) := \int_0^{r_0} 
\biggl[ \left(\bv' + \oh(\frac{r}{R} - \bw')^2 \right)^2 + \alpha_s h^{-2} |\bw|^2 + 
h^2 \left|\bw''-\frac1R\right|^2 + \W(\bv / r)\biggr] r \dr.
\end{equation}
Indeed, the term $\alpha_s h^{-2} |\bw|^2$ prefers $\bw = 0$, and $F_h(\bv,0)$ differs 
from $F_0(\bv)$ by a term that's $O(h^2)$. This rather formal calculation will be justified
by Proposition~\ref{prop}, which shows (among other things) that $|\min F_0 - \min F_h| \le Ch^2$.
Thus (by the triangle inequality) we can replace $\min F_0$ by $\min F_h$ in the definition of the
excess energy. The functional $F_h$, in turn, is a lower bound for the elastic energy. This will be
explained in Section~\ref{sec:lb}; briefly, it follows from Plancherel's formula ($\bv$ represents 
the azimuthal average of the radial displacement, while $\bw$ represents the azimuthal average 
of $\xi + |x|^2/2R$). Incidentally, when wrinkling is not expected (i.e. if the~minimizer $(\bv_h,\bw_h)$ of $F_h$ satisfies $\bv_h \ge -2\alpha^{1/2}r$ for $r \in (0,r_0)$), 
the arguments in Section~\ref{sec:lb} show that the~radial extension of $(\bv_h,\bw_h)$ minimizes 
$E_h$, so that $\min E_h = \min F_h$. 

To explain how the minimizer of the effective functional provides information about the~length scale of wrinkling we must say more about our analysis of energetic cost of wrinkling
(the~calculation behind $\W$, defined by~(\ref{wrel})). Its starting point is an expression
for 
$$
W_r = \mbox{azimuthal strain} + \mbox{substrate energy} + \mbox{azimuthal bending energy}
$$
integrated over the circle $|x|=r$ (for the precise definition see~\eqref{Wenergy}). It depends
not only on $\eta$ (the argument of $\W$) but also on the normal displacement of the sheet (at 
radius $r$), viewed as a~function of $\theta$. The formula (\ref{wrel}) is obtained by 
optimizing the displacement. But this calculation gives more than just the optimum; in particular,
it reveals the extra energy $W_r - \W$ associated with a non-optimal choice of the wrinkling. 
The exact expression is~(\ref{e2.28}), but the main point is this: if $a_k(r)$ is the $k$th
Fourier coefficient of $\xi(r,\cdot)$ then $W_r - \W$ is at least 
$$
\sum_{k \neq 0} a_k^2 \left(\frac{k}{r}\right)^2 
\left(\frac{h|k|}{r} - \alpha_s^{\oh} \frac{r}{|k|h}\right)^2.
$$
This suggests wrinkling with frequency $k$ such that 
$\frac{h|k|}{r} = \alpha_s^{1/2} \frac{r}{|k|h}$, i.e. $|k|=\alpha_s^{1/4} r/h$, 
a choice that makes the length scale of wrinkling independent of $r$ and 
proportional to $h$. (Essentially the~same calculation can be found 
in~\cite{HohlfeldDavidovitch}.) Our conclusions about the length scale of wrinkling are
not a surprise. Indeed, since this length scale is set by the competition between 
the bending and substrate terms, it should not depend on $r$. Since the wrinkles
are arranged radially, it follows that the number per circle (the wavenumber $|k|$) 
should increase linearly with $r$. 

The preceding calculation considers each circle separately. But we have already mentioned
that $E_h$ is strictly above the minimum of $F_0$, because the radial variation of the 
wrinkling pattern costs additional (``excess'') energy, whose magnitude is at least of
order $h$. So our identification of an ``optimal'' $k$ should not be taken literally. Rather, 
we expect a wrinkling pattern such that $W_r - \W$ is of order $h$. Thus: we
expect that $|\frac{hk}{r} - \alpha_s^{1/2} \frac{r}{kh}|$ should be at most of order
$h^{1/2}$ on the support of $a_k$. 

In Section~\ref{sec:ub} we provide an explicit wrinkling pattern whose excess energy is 
approximately linear in $h$. The observation in the last paragraph -- that the active $k$
at radius $r$ can be a certain distance from $\alpha_s^{1/4} r/h$ -- turns out to be crucial. 
Since changing the wrinkling pattern costs elastic energy, the successful construction avoids
changing it too often. Very roughly speaking, the active $k$ is a function of $r$ which 
approximates the linear function $\alpha_s^{1/4} r/h$ but takes only values that are integer
multiples of $h^{-a}$ with $a \approx 1/2$. An account of why such a construction seems
necessary is provided at the beginning of Section~\ref{sec:ub}.

\BR6
One feature of our story is a little bit surprising. A wrinkled surface resists curvature in the~direction parallel to the wrinkles. In some recent studies of wrinkling this effect plays a crucial role in determining the local length scale \cite{Paulsen++,taffetani}. Our problem is different from the ones in those studies; the key differences (as noted earlier) are that (1) the coefficient of our substrate term scales like $\alpha_s h^{-2}$ with $\alpha_s$ fixed and positive as $h \rightarrow 0$, and therefore (2) the leading-order behavior in our setting is not obtained from tension-field theory or a relaxed problem, but rather by our effective functional. It is nevertheless natural to wonder at what order the effect of curvature is felt as $h \rightarrow 0$.  Since our upper and lower bounds on the leading-order correction
(the scaling of $\min E_h - \min F_0$ with respect to $h$) do not match, our results are consistent with the possibility that curvature effects might play a role at this order. As for the length scale of wrinkling: while our results show that it is approximately constant (this is required to get the correct leading-order behavior), its small deviations from constancy might be influenced by effects that do not enter our analysis, such as the extra cost of wrinkling in a curved environment. 
\ER

\section{The lower bound}\label{sec:lb}

In the rest of the paper $\lesssim$ will stand for $\le C$, where $C > 0$ is a generic constant depending on $\alpha_s$, $r_0$, $R$. (The implicit constant is often dimensional.)

The radial symmetry of the domain $\Omega$ and of the substrate allows for convenient representation of the energy~\eqref{energy} in the radial coordinates $(r,\theta)$:
\begin{align}\label{e1.1}
 E_h(u,\xi) &= \int_0^{r_0} \biggl( \int_0^{2\pi} \biggl|\partial_r u_r + \frac{(\partial_r \xi)^2}{2} \biggl|^2 + \biggl|\frac{\partial_\theta u_\theta}{r} + \frac{u_r}{r} + \frac{(\partial_\theta \xi)^2}{2r^2} \biggl|^2 \nonumber
 \\
 &\hskip10em\relax + \frac{h^2}{r^4} |\partial_{\theta\theta} \xi|^2 + h^2 |\partial_{rr} \xi|^2 + \frac{\alpha_s}{h^2} \biggl|\xi + \frac{r^2}{2R}\biggl|^2 \dt \biggr) r \dr
 \\ \nonumber
 & \quad + \int_0^{r_0} \biggl( \int_0^{2\pi} 2\biggl|\frac{1}{2r} \partial_\theta u_r + \oh \partial_r u_\theta - \frac{1}{2r} u_\theta + \frac{1}{2r} \partial_r \xi \partial_\theta \xi\biggl|^2 + \frac{2h^2}{r^2} |\partial_{\theta r} \xi|^2 \dt\ \biggr) r \dr.
\end{align}

Motivated by the fact that in thin sheets (i.e. $h \ll 1$ and so $\alpha_s h^{-2} \gg 1$) the quantity $\xi + \frac{|x|^2}{2R}$ should be very small, instead of $\xi$ we consider $w$ defined by
\begin{equation}\nonumber
 w(r,\theta) = \xi(r,\theta) + \frac{r^2}{2R}.
\end{equation}
Since for any $f \in L^2(0,2\pi)$ we have
 \begin{equation}\label{e1.3}
 \fint_0^{2\pi} |f|^2 = |\bar f|^2 + \fint_0^{2\pi} |f - \bar f|^2,
\end{equation}
where $\bar f = \fint_0^{2\pi} f$, we see that for $\bar u_r(r) := \fint_0^{2\pi} u_r(r,\theta) \dt$ we have
\begin{equation}\label{e1.4}
\begin{aligned}
 \fint_0^{2\pi} \biggl|\frac{\partial_\theta u_\theta}{r} + \frac{u_r}{r} + \frac{(\partial_\theta w)^2}{2r^2}\biggl|^2 \dt
&=
 \left| \frac{\bar u_r}{r} + \fint_0^{2\pi} \frac{(\partial_\theta w)^2}{2r^2} \dt \right|^2
 \\
 &\quad + \fint_0^{2\pi} \biggl|\frac{\partial_\theta u_\theta}{r} + \frac{u_r}{r} + \frac{(\partial_\theta w)^2}{2r^2} - \frac{\bar u_r}{r} - \fint_0^{2\pi} \frac{(\partial_\theta w)^2}{2r^2} \dt \biggr|^2 \dt.
\end{aligned}
\end{equation}
Defining $\bw(r) := \fint_0^{2\pi} w(r,\theta) \dt$, we use definition of $w$ and~\eqref{e1.3} twice to write
\begin{equation}
\begin{aligned}\label{e1.5}
 \fint_0^{2\pi} \biggl|\partial_r u_r + \frac{(\partial_r \xi)^2}{2}\biggl|^2 \dt &= \biggl|\partial_r \bar u_r + \frac 12 \left(\frac{r}{R} - \partial_r \bw\right)^2 + \fint_0^{2\pi} |\partial_r (\bw - w)|^2 \dt \biggr|^2
\\
 &\quad + \fint_0^{2\pi} \biggl| \partial_r (u_r-\bar u_r) + \frac{(\partial_r \xi)^2}{2} - \fint_0^{2\pi} \frac{(\partial_r \xi)^2}{2}\dt\biggr|^2 \dt.
\end{aligned}
\end{equation}
Plugging~\eqref{e1.4} and~\eqref{e1.5} into~\eqref{e1.1} we get that
\begin{equation}\label{e2.12}
\begin{aligned}
 &E_h(u,\xi)
\\
 &= \int_0^{r_0} \biggl( \biggl|\partial_r \bar u_r + \frac 12 \left (\frac rR - \partial_r \bw\right)^2 + B(r)\biggl|^2 + W_r\left(\frac{\bar u_r}{r},w(r,\cdot)\right) + \frac{\alpha_s}{h^2} |\bar w|^2 + h^2 \biggl|\partial_{rr} \bw - \frac 1R\biggl|^2 \biggr) r \dr
\\ &\quad + R_h(u,\xi),
\end{aligned}
\end{equation}
where $B(r) := \fint_0^{2\pi} |\partial_r(\bar w - w)|^2 \dt$,
\begin{equation}\label{Wenergy}
 W_r(\eta,\xi) :=
  \biggl| \eta + \fint_0^{2\pi} \frac{1}{2r^2} (\partial_\theta \xi)^2 \dt\biggl|^2 + h^2 r^{-4} \fint_0^{2\pi} |\partial_{\theta\theta} \xi|^2 \dt + \alpha_s h^{-2} \fint_0^{2\pi} \biggl|\xi - \fint_0^{2\pi} \xi\biggl|^2 \dt,
\end{equation}
and
\begin{equation}\label{Rh}
\begin{aligned}
 R_h(u,\xi) &= \int_0^{r_0} \fint_0^{2\pi} \biggl| \partial_r (u_r-\bar u_r) + \frac{(\partial_r \xi)^2}{2} - \fint_0^{2\pi} \frac{(\partial_r \xi)^2}{2}\biggl|^2 \dt\ r \dr 
 \\ &\quad + \int_0^{r_0} \fint_0^{2\pi} \biggl|\frac{\partial_\theta u_\theta}{r} + \frac{u_r}{r} + \frac{(\partial_\theta w)^2}{2r^2} - \frac{\bar u_r}{r} - \fint_0^{2\pi} \frac{(\partial_\theta w)^2}{2r^2} \biggr|^2 \dt\ r \dr
\\
& \quad + \int_0^{r_0}\! \biggl( \fint_0^{2\pi}
\!\!\frac 12 \biggl|\frac{1}{r} \partial_\theta u_r\! +\! r \partial_r (\frac{u_\theta}{r})\! +\! \frac{1}{r} \partial_r \xi \partial_\theta \xi\biggl|^2\! + h^2 |\partial_{rr}(w - \bw)|^2\! +\! \frac{2h^2}{r^2} |\partial_{\theta r} \xi|^2 \dt \biggr) r \dr.
\end{aligned}
\end{equation}

To show that $\epsilon \ge c_0 h$, we need to gather some properties of the functionals $F_h$:

\begin{proposition}\label{prop}
Under the condition $\alpha_s < 2^{-8}(r_0/R)^4$, the functional $F_h$, defined in~\eqref{fh}, admits a unique minimizer $(\bv_h,\bw_h)$. Moreover, denoting $\sigma_h(r) := \bv_h' + \frac 12 (\frac rR - \bw_h')^2$, there exists $C=C(\alpha_s,r_0,R)$ such that 
\begin{gather}\label{fhfh}
  |F_0(\bv_0) - F_h(\bv_h,\bw_h)| \le C h^2,\\
\begin{gathered}\label{e2.15}
  \int_0^{r_0} \bw_h^2 r \dr \le C h^4, \quad \int_0^{r_0} \bw_h'^2 r \dr \le C h^2,
 \quad \int_0^{r_0} |\bv_h - \bv_0|^2 r \dr \le C h^2,\\
\sigma_h(r) = 4\alpha_s^{1/2}(r_0/r-1) \quad \emph{for } r \in ((2r_w+r_0)/3,r_0),
 \end{gathered}
\end{gather}
where $r_w = (16 \alpha_s^{\frac12} r_0 R^2)^{\frac13}$ and $\bv_0$ is defined in~\eqref{vone}.
 Finally, for any non-negative $B \in L^1(0,r_0)$ and any $(\bv,\bw)$ such that
 \begin{equation}\nonumber
  e\! :=\! \int_0^{r_0} \biggl[ \biggl(\bv' + \oh\left(\frac{r}{R} - \bw'\right)^2\! +\! B\! \biggl)^2\! + \alpha_s h^{-2} |\bw|^2\! + \!h^2 \biggl|\bw''\!-\!\frac1R\biggl|^2\! + W_{\emph{rel}}(\bv / r)\biggl] r \dr\! -\! F_h(\bv_h,\bw_h) \le 1,
 \end{equation}
  we have
 \begin{equation}\label{e2.1}
  \int_0^{r_0} \left(2\sigma_h(r) B(r) +  (\bv - \bv_h)_+^2\right)  r \dr \le C e.
 \end{equation} 
\end{proposition}
The proof of Proposition~\ref{prop} will be given in Section~\ref{sec:prop}. 
Since $R_h \ge 0$, relation~\eqref{e2.12} and the fact $W_r(\eta,\xi) \ge \W(\eta)$, together with~\eqref{e2.1} applied with $\bv := \bar u_r$ imply a lower bound 
\begin{multline*}
 \inf_{u,\xi}
 \int_{0}^{r_0} \biggl( 2\sigma_h(r) B(r) + W_r\left(\frac{\bar u_r}{r},w(r,\cdot)\right) - \W\left(\frac{\bar u_r}{r}\right) + \left(\bar u_r - \bv_h\right)_+^2 \biggr) r \dr 
 \\
 \lesssim \inf_{u,\xi} E_h(u,\xi) - F_h(\bv_h,\bw_h),
\end{multline*}
which using~\eqref{fhfh} can be postprocessed into
\begin{equation}\label{e2.18}
  \inf_{u,\xi}
 \int_{0}^{r_0} \biggl( 2\sigma_h(r) B(r) + W_r\left(\frac{\bar u_r}{r},w(r,\cdot)\right) - \W\left(\frac{\bar u_r}{r}\right) + \left(\bar u_r - \bv_h\right)_+^2 \biggr) r \dr \le C\left( \epsilon + h^2\right),
\end{equation}
where we recall that $\epsilon = \inf_{u,\xi} E_h(u,\xi) - F_0(\bv_0)$ and $B(r) = \fint_0^{2\pi} |\partial_r(\bar w - w)|^2 \dt$. To obtain the~desired lower bound for the left-hand side in~\eqref{e2.18}, for $r \in (0,r_0)$ let $\{ a_k(r) \}_{k \in \mathbb{Z}}$ be the Fourier coefficients of $w(r,\cdot)$. To simplify notation, we will drop the $r$-dependence in $a_k(r)$. Then
\begin{align*}
 B(r) &= \sum_{k \neq 0} \left(\partial_r a_k\right)^2, \\
 W_r(\eta,w(r,\cdot)) &= \biggl|\eta + \frac{1}{2r^2} \sum_{k \in \mathbb{Z}} a_k^2 k^2\biggr|^2 + \sum_{k \in \mathbb{Z} \setminus \{0\}}\! a_k^2 k^2 r^{-2} \biggl( \left(\frac{kh}{r}\right)^2 + \alpha_s \left(\frac{kh}{r}\right)^{-2} \biggr).
\end{align*}

The lower bound $\epsilon \ge c_0h$ will be a consequence of the following lemma:

{
\renewcommand{\r}{\varrho}
\begin{lemma}\label{lm1}
 Let $\delta > 0$, and let $\r_0, \r_1 \in (0,r_0)$ be such that 
$\r_0 < \r_1 < \sqrt{2} \r_0$ and  
 \begin{equation}\label{e5.2}
  \frac{\bar u_r(\r_i)}{\r_i} \le -2\alpha_s^{\oh} - \delta, \quad i=0,1.
 \end{equation}
 Then we have a lower bound of the form
 \begin{multline} \label{lemma1-conclusion}
 \sum_{i=0}^1 \biggl(W_r\left(\frac{\bar u_r(\r_i)}{\r_i},w\right) - W_\emph{rel}\left(\frac{\bar u_r(\r_i)}{\r_i}\right)\biggr) + \fint_{\r_0}^{\r_1} B(r) \dr 
  \\
  \ge
  \min\biggl( \frac{\delta^2}{4}, \frac{\delta}{8} \left(\frac{l}{\r_0}\right)^2 \alpha_s^\oh, \frac{\delta}{2} \biggl( \frac{32}{\alpha_s^{1/2}} \left(\frac{\r_0}{l} \right)^2 + 8{\alpha_s^{1/2}} \left( \frac{l}{h} \right)^2 \biggr)^{-1}\biggl),
 \end{multline}
 where $l := \r_1 - \r_0$.
\end{lemma}

Condition~\eqref{e5.2} expresses the fact that the sheet is expected to wrinkle at circles $r=\r_i, i=0,1$. 

\begin{proof}[Proof of Lemma~\ref{lm1}.]
 Recall that $l = \r_1 - \r_0$ and let $k_i := \alpha_s^{1/4} \r_i / h$. As explained in Section~\ref{sec:model}, $k_i$ is the optimal wrinkling wavenumber at position $\r_i$, i.e. the one which minimizes \BR6 the \ER energetic cost of wrinkling.
 By reorganizing the definition of $W_r$ we get 
 \begin{equation}\nonumber
  W_r(\eta,w(r,\cdot)) = \left|\eta + \oh \sigma\right|^2 + 2\alpha_s^{\oh} \sigma + \sum_{k \neq 0} a_k^2 \left(\frac{k}{r}\right)^2 \left(\frac{h|k|}{r} - \alpha_s^{\oh} \frac{r}{|k|h}\right)^2,
 \end{equation}
 where
$\sigma(r) := \fint_0^{2\pi} \frac{1}{r^2} |\partial_\theta \xi|^2 \, d\theta =
 \sum_{k \neq 0} a_k^2 (\frac{k}{r})^2 \ge 0$ represents the arclength wasted by wrinkling.
Since in all the quantities only the modulus of $k$ appears, to simplify the notation we assume that $a_k = 0$ for $k < 0$. From the definition of $\W$ we get for $\eta < -2\alpha_s^{1/2}$ the following identity:
 \begin{equation}\label{e2.28}
  W_r(\eta,w(r,\cdot)) - \W(\eta) = \frac14 \left|\sigma + 2\left(\eta + 2\alpha_s^\oh\right)\right|^2 + \sum_{k > 0} a_k^2 \left(\frac{k}{r}\right)^2 \left(\frac{hk}{r} - \alpha_s^{\oh} \frac{r}{kh}\right)^2.
 \end{equation}
If either $\sigma(\r_0) \le \delta$ or $\sigma(\r_1) \le \delta$, then the 
conclusion of the lemma follows easily. Indeed, in this case we can 
write~\eqref{e5.2} as $\eta + 2\alpha_s^{1/2} \le -\delta$ with
$\eta = \bar u_r(\rho_i)/\rho_i$ for $i=0$ or $i=1$; it follows that 
$\frac14 |\sigma + 2(\eta + 2\alpha_s^\oh)|^2 \ge \frac 14 \delta^2$, which implies the
conclusion of the lemma. So we may assume for the~rest of the proof that 
$\sigma(\r_0) \ge \delta$ and $\sigma(\r_1) \ge \delta$.

 Let $K := \frac{k_1 - k_0}{2} = \frac l2 \frac{\alpha_s^{1/4}}{h}$. We consider two cases. If the energy at $\r_0$ is not concentrated near the optimal wavenumber $k_0$, meaning
 \begin{equation}\label{e2.29}
  \sum_{|k - k_0| > K} a_k^2 \left(\frac{k}{\r_0}\right)^2 \ge \frac 12 \sigma(\r_0),
 \end{equation}
 then by~\eqref{e2.28}
 \begin{align*}
  W_{\r_0}\left(\frac{\bar u_r(\r_0)}{\r_0},w(\r_0,\cdot)\right) - \W \left(\frac{\bar u_r(\r_0)}{\r_0}\right)
  &\uoset{\eqref{e2.29}}{}{\ge} \sum_{|k-k_0| > K} a_k^2 \left(\frac{k}{\r_0}\right)^2 \left(\frac{hk}{\r_0} - \alpha_s^{\oh} \frac{\r_0}{kh}\right)^2
  \\
  &\uoset{\eqref{e2.29}}{}{\ge} \sum_{|k-k_0| > K} a_k^2 \left(\frac{k}{\r_0}\right)^2 \left(\frac{h}{\r_0} |k-k_0|\right)^2
  \\
  &\overset{\eqref{e2.29}}{\ge} \frac{\sigma(\r_0)}{8} \left(\frac{l}{\r_0}\right)^2 \alpha_s^\oh.
 \end{align*}
 In the other case we have
 \begin{align*}
 \frac{\sigma(\r_0)}{2} &\le \sum_{|k - k_0| \le K} a_k^2(\r_0) \left(\frac{k}{\r_0}\right)^2 
 \\
 &\le 2 \sum_{|k - k_0| \le K} a_k^2(\r_1) \left(\frac{k}{\r_0}\right)^2 + 2l\int_{\r_0}^{\r_1} \sum_{|k - k_0| \le K} (\partial_r a_k(r))^2 \left(\frac{k}{\r_0}\right)^2 \dr,
 \end{align*}
 where we used an elementary inequality $f^2(l) \le 2f^2(0) + 2l \int_0^l f'^2$. To estimate the first term on the right-hand side, we observe that $(k/\r_0)^2 \le 2(k/\r_1)^2$ and that for $
 |k-k_0|<K$ we have $|hk/\r_1 - \alpha_s^\oh \r_1/(hk)| \ge \frac{h}{\r_1} |k - k_1| \ge \alpha_s^{1/4} \frac{l}{2\r_1}$.
 Hence
 \begin{equation}\nonumber
  2 \sum_{|k - k_0| \le K} a_k^2(\r_1) \left(\frac{k}{\r_0}\right)^2 \le \frac{16}{\alpha_s^{1/2} } \left(\frac{\r_1}{l}\right)^2 \sum_{|k - k_0| \le K} a_k^2(\r_1) \left(\frac{k}{\r_1}\right)^2 \left(\frac{hk}{\r_1} - \alpha_s^\oh \frac{\r_1}{hk}\right)^2.
 \end{equation}
 For the second term we observe that for $|k - k_0| \le K$ we have $|k/\r_0|^2 \le 4\alpha_s^\oh/h^2$, and so
 \begin{equation}\nonumber
  2l\int_{\r_0}^{\r_1} \sum_{|k - k_0| \le K} (\partial_r a_k(r))^2 \left(\frac{k}{\r_0}\right)^2 \dr \le 8\alpha_s^\oh \frac{l^2}{h^2} \fint_{\r_0}^{\r_1} B(r) \dr.
 \end{equation}
 Summing these two and using that $\r_1^2 \le 2\r_0$ yields
 \begin{equation}\nonumber \frac{\delta}{2} \le \frac{\sigma(\r_0)}{2} \le \biggl( \frac{32}{\alpha_s^{1/2}} \left(\frac{\r_0}{l} \right)^2 + 8{\alpha_s^{1/2}} \left( \frac{l}{h} \right)^2 \biggr) \cdot \biggr(\mbox{LHS of }~(\ref{lemma1-conclusion}) \biggr),
 \end{equation}
 which concludes the proof of the lemma.
\end{proof}

\begin{proof}[Proof of the lower-bound half of Theorem~\ref{thm}, i.e. the estimate $\epsilon \ge c_0 h$.]
Let $I$ be any interval of length $\alpha_s^{-1/4} r_w^{1/2} h^{1/2}$ such that $I \subset [(2r_w+r_0)/3,(r_w+2r_0)/3]$
(such intervals exist, provided that the constant $\h$ in the 
statement of Theorem~\ref{thm} is chosen sufficiently small). 
The condition $I \subset [(2r_w+r_0)/3,(r_w+2r_0)/3]$ assures that $I$ stays away from 
the inner and outer boundaries of the wrinkled region (which are at $r_w$ and $r_0$
respectively). One verifies using~\eqref{vone} the existence of a constant $C_2>0$ such that
$\frac{\bv_0(r)}{r} + 2\alpha_s^{1/2} \le -2C_2$ for all $r \in I$ (and all such intervals
$I$). It follows that 
$\frac{\bar u_r(r)}{r} + 2\alpha_s^{1/2} \le -2C_2 + \frac{(\bar u_r(r) - \bv_0(r))_+}{r}$. 
Hence at each $r \in I$ we have either $\frac{\bar u_r(r)}{r} + 2\alpha_s^{1/2} \le -C_2$ 
or $\frac{(\bar u_r(r) - \bv_0(r))_+}{r} \ge C_2$. 
Writing $I = (a,b)$, let us define $M_0 \subset (a,(2a+b)/3)$, $M_1 \subset ((a+2b)/3,b)$ to be maximal sets such that for $r \in M_0 \cup M_1$
\begin{equation}\label{e4.1}
 \bar u_r(r) \le -2\alpha_s^{1/2}r - C_2r.
\end{equation}
We distinguish two cases. If $|M_0| < |I|/6$ or $|M_1| < |I|/6$, the complement $I \setminus (M_0 \cup M_1)$ has measure at least $|I|/6$ and inside this complement~\eqref{e4.1} is false, and so by above considerations $\frac{(\bar u_r(r) - \bv_0(r))_+}{r} \ge C_2$. Using the triangle inequality, this gives $(\bar u_r - \bv_h)_+ + |\bv_h - \bv_0| \ge (\bar u_r - \bv_h)_+ + (\bv_h - \bv_0)_+ \ge (\bar u_r-\bv_0)_+ \ge C_2r$, which then implies $(\bar u_r - \bv_h)_+^2 + |\bv_h - \bv_0|^2 \gtrsim 1$, which in turn after integration yields $\int_I ( (\bar u_r - \bv_h)_+^2 + |\bv_h - \bv_0|^2 ) \gtrsim |I|$.

In the other case, i.e. if both $|M_0| \ge |I|/6$ and $|M_1| \ge |I|/6$, for $i=0,1$ we can choose $\rho_i = \arg\min_{\rho \in M_i} W_r(\frac{\bar u_r(\rho)}{\rho},w(\rho_i,\cdot)) - \W(\frac{\bar u_r(\rho)}{\rho})$, and apply Lemma~\ref{lm1} to get
\begin{align*}
 \int_I W_r\left(\frac{\bar u_r(r)}{r},w(r,\cdot)\right) &- \W\left(\frac{\bar u_r(r)}{r}\right) + B(r) \dr \gtrsim
 \\
 &|I| \min\biggl( 1, \left(\frac{l}{\r_0}\right)^2 \alpha_s^\oh, \biggl( \frac{32}{\alpha_s^{1/2}} \left(\frac{\r_0}{l} \right)^2 + 8{\alpha_s^{1/2}} \left( \frac{l}{h} \right)^2 \biggr)^{-1}\biggl),
\end{align*}
where the fact that $l = \rho_1 - \rho_0 \sim \alpha_s^{-1/4}r_w^{1/2} h^{1/2}$ implies that the right-hand side in the previous estimate is at least of order $h$. Combining the two cases we get
\begin{equation}\nonumber
 \int_I (\bar u_r - \bv_h)_+^2 + |\bv_h - \bv_0|^2 + W_r\left(\frac{\bar u(r)}{r},w(r,\cdot)\right) - \W\left(\frac{\bar u(r)}{r}\right) + B(r) \dr \gtrsim |I| h.
\end{equation}
Finally, we use estimate~\eqref{e2.18} together with the fact that $\sigma_h \gtrsim 1$ and $r \sim 1$ inside $[2r_w + r_0)/3, (r_w + 2r_0)/3]$, and the estimate on the difference $\bv_h - \bv_0$, and cover at least half of $[2r_w + r_0)/3, (r_w + 2r_0)/3]$ with disjoint intervals like the interval $I$ considered above to get the~desired lower bound
 $\epsilon \gtrsim h$.
\end{proof}

\section{Proof of Proposition~\ref{prop}}\label{sec:prop}

\begin{proof}[Proof of Proposition~\ref{prop}.]
We begin by showing that $F_h$ achieves its minimum. Indeed, replacing the membrane
term in \eqref{fh} by its positive part -- that is, replacing the first term under the~integral by $\left( \bv' + \oh(\frac{r}{R}-\bw')^2 \right)_+^2 $ -- makes the~problem convex. This modified problem achieves its minimum by the direct method of the
calculus of variations. We claim that its minimizer $(\bv_h,\bw_h)$ also minimizes 
the original functional $F_h$. This follows from the fact that for 
$(\bv_h,\bw_h)$ the expression $\bv_h' + \oh(\frac{r}{R} - \bw_h')^2$ must be 
non-negative almost everywhere (since otherwise one can easily modify 
$(\bv_h,\bw_h)$ to obtain a competitor with strictly smaller energy). 
Heuristically this is clear, since the relation $\bv_h' + \oh(\frac{r}{R} - \bw_h')^2 \ge 0$
says simply that the sheet is in tension in the radial direction. This argument also
shows that the minimizer $(\bv_h,\bw_h)$ is unique, since the modified functional 
is everywhere convex, and strictly convex whenever $\bv' + \oh(\frac{r}{R}-\bw')^2 > 0$.  

At risk of redundancy we sketch another proof that the minimizer exists, by applying the~direct method of the calculus of variations directly to $F_h$. Indeed, since 
$h$ is strictly positive the~functional $F_h$ controls the second derivative of $\bar w$.
Therefore for a minimizing sequence $\bar w'$ would converge strongly, and this can be 
used to pass to the limit in the first term. 

To obtain information about the form of $(\bv_h,\bw_h)$, we first set $w = 0$ and minimize $F_h(v,0)$, which is the same as minimizing $F_0$. This can be done by guessing that $(0,r_0)$ splits into two non-trivial intervals: $(0,r_w)$, where $\bv(r)/r \ge -2\alpha_s^{1/2}$, meaning in this region the sheet will not wrinkle (hence we use that $\W(\eta) = \eta^2$ here), and $(r_w,r_0)$, where the sheet wrinkles. Solving the~resulting ordinary differential equation leads to the explicit form~\eqref{vone} for the minimizer $\bv_0$ of $F_0$. 

Using the optimality properties for $(\bv_h,\bw_h)$ and $\bv_0$ we get that $DF_h(\bv_h,\bw_h) = 0$ and $D_vF_h(\bv_0,0)=0$. In addition, for any smooth test function $\varphi$ we have for some function $g$: 
\begin{equation*}
D_wF_h(\bv_0,0)\varphi = \int_0^{r_0} -2\left(\bv_0 + \frac{1}{2} \frac{r^2}{R^2} \right) \frac{r}{R} \varphi' r \dr = \int_0^{r_0} g(r) \varphi(r) r \dr.
\end{equation*}
We can use~\eqref{vone} to get explicit formula for $g$, and in particular to show that $g$ is bounded by some constant $C_g$. Here by $DF_h = 0$ and $D_v F_h = 0$ we mean that the Gateaux derivative (i.e. directional derivative) of the functional $F_h$ vanishes for every direction $(v,w)$ and $(v,0)$, respectively. 

Recall that $\sigma_h(r) = \bv_h'(r) + \frac{1}{2}(\frac{r}{R} - \bw_h(r)')^2$ (see~\eqref{vone}) and define $\sigma_0(r) := \bv_0'(r) + \frac{1}{2} \frac{r^2}{R^2}$. Then using a simple identity for $f(x,y) = (x + \frac 12 y^2)^2$ of the form
\begin{equation}\label{fident}
 f(x,y) - f(a,b) - Df(a,b)\cdot (x-a,y-b) = \left( \left(x+\tfrac12 y^2\right) - \left(a+\tfrac 12 b^2\right) \right)^2 + \left(a+\tfrac12 b^2\right) (y-b)^2,
\end{equation}
we get using Taylor's theorem with the remainder in the integral form (here we use that $\W'$ is absolutely continuous)
\begin{align}\label{e2.37}
 F_h(\bv_0,&0) - F_h(\bv_h,\bw_h) - DF_h(\bv_h,\bw_h)\cdot(\bv_0-\bv_h,-\bw_h)
\\ \nonumber
&= \int_0^{r_0} \biggl[ (\sigma_0 - \sigma_h)^2 + \sigma_h (\bw_h')^2 + \alpha_s h^{-2} \bw_h^2 + h^2 \bw_h''^2 + \int_{\bv_h/r}^{\bv_0/r} \W''(\xi) (\bv_0/r - \xi) \ud \xi \biggr] r \dr,
\\ \nonumber
F_h(\bv_h,&\bw_h) - F_h(\bv_0,0) - DF_h(\bv_0,0)\cdot(\bv_h-\bv_0,\bw_h)
\\ \nonumber
&= \int_0^{r_0} \biggl[ (\sigma_0 - \sigma_h)^2 + \sigma_0 (\bw_h')^2 + \alpha_s h^{-2} \bw_h^2 + h^2 \bw_h''^2 + \int_{\bv_0/r}^{\bv_h/r} \W''(\xi) (\bv_h/r - \xi) \ud \xi
\biggr] r \dr.
\end{align}
We have already observed that $DF_h(\bv_h,\bw_h) = 0$, $D_vF_h(\bv_0,0)=0$, and also $|D_wF_h(\bv_0,0)\varphi| \le C_g \int_0^{r_0} |\varphi(r)| r \dr$.
We add the previous two relations and use these facts about $DF_h$ to obtain
\begin{equation}\label{e2.36}
\begin{aligned}
\int_0^{r_0} \biggl[ 2(\sigma_0 - \sigma_h)^2\! + (\sigma_h + \sigma_0) (\bw_h')^2\! + 2\alpha_s h^{-2} \bw_h^2 + 2h^2 \bw_h''^2
\!+\! \int_{\bv_h/r}^{\bv_0/r}\!\!\! \W''(\xi) &(\bv_0/r - \bv_h/r) \ud \xi \biggr] r \dr
\\
&\le C_g \int_0^{r_0} \bw_h r \dr.
\end{aligned}
\end{equation}

First observe that both $\sigma_h \ge 0$ and $\sigma_0 \ge 0$ -- the first relation comes from the Euler-Lagrange equation while the second can be derived either by direct computation or also through the~Euler-Lagrange equation. Next we observe that $\W'' \ge 0$ a.e., which follows from convexity of $\W$ and can be also explicitly computed. 
Hence, using that $\sigma_h \ge 0, \sigma_0 \ge 0, \W'' \ge 0$ we see that all the terms on the left-hand side of~\eqref{e2.36} are non-negative, in particular $\int_0^{r_0} 2\alpha_s h^{-2} \bw_h^2 r \dr \lesssim \int_0^{r_0} \bw_h r \dr$, which by H\"older's inequality implies $\int_0^{r_0} \bw_h^2 r \dr \lesssim \alpha_s^{-2} h^4$, which in turn shows that the right-hand side in~\eqref{e2.36} is at most of order $h^2$. 

Since the right-hand side in~\eqref{e2.37} is bounded by the left-hand side in~\eqref{e2.36}, we see that 
\begin{equation}\nonumber
 F_h(\bv_0,0) - F_h(\bv_h,\bw_h) - DF_h(\bv_h,\bw_h)\cdot(\bv_0-\bv_h,-\bw_h) \lesssim h^2. 
\end{equation}
Using that $DF_h(\bv_h,\bw_h)=0$, we get that $F_h(\bv_0,0) - F_h(\bv_h,\bw_h) \lesssim h^2$. The trivial 
fact that $F_h(\bv_0,0) \ge F_h(\bv_h,\bw_h)$ (which holds since $F_h(\bv_h,\bw_h)$ is the minimum over a 
larger class of test functions) permits us to conclude~\eqref{fhfh}.

The right-hand side in~\eqref{e2.36} being $\lesssim h^2$ also implies $\int_0^{r_0} \bw_h''^2 r \dr \lesssim 1$, which by an interpolation inequality $\| \nabla f \|_{L^4(B_{r_0})} \lesssim \| f \|_{L^2(B_{r_0})}^{1/4} \| \nabla^2 f \|_{L^2(B_{r_0})}^{3/4} + \| f \|_{L^2(B_{r_0})}$ 
implies $\int_0^{r_0} \bw_h'^4 r \dr \lesssim h^2$. We recall the definition of $\sigma_h,\sigma_0$, and see that $\int_0^{r_0} |\sigma_h - \sigma_0|^2 r \dr \lesssim h^2$ implies $\int_0^{r_0} |\bv_h' - \bv_0'|^2 r \dr \lesssim h^2$. Since $\W(\eta)$ is uniformly convex in the regime $\eta \ge -2\alpha_s^{1/2}$, meaning in the region where $\bv_0(r)/r$ lies well within this regime we get control in $L^2$ on the difference $(\bv_0(r) - \bv_h(r))/r$, by a standard embedding of $W^{1,2}$ into $L^p$ in $\R^2$, for any $p \in [1,\infty)$ we get $(\int_0^{r_0} |\bv_0 - \bv_h|^p r \dr )^{1/p} \le C(p) h$, and~\eqref{e2.15} follows. In addition, since for $r \in J:=(\min(r_w,r_0/2),r_0)$ we have $r \gtrsim 1$, we also see that $|\bv_h(r) - \bv_0(r)| \lesssim h$ when $r \in J$.

To obtain the formula for $\sigma_h$, we observe that $\bv_0(r) \le -2\alpha_s^{1/2}$ for $r \in (r_w,r_0)$ together with the above $L^\infty$ bound on $\bv_h - \bv_0$ implies, assuming $h$ is small enough, that $\bv_h(r) \le -2\alpha_s^{1/2}$ for $r \in ((2r_w+r_0)/3,r_0)$, which in turn gives $\W'(\bv_h(r)/r) = -4\alpha_s^{1/2}$ for $r \in ((2r_w+r_0)/3,r_0)$. Solving the Euler-Lagrange equation $(\sigma_h(r)r)' = \W'(\bv_h(r)/r)$ together with the corresponding boundary condition $\sigma(r_0)=0$, we immediately obtain the formula for $\sigma_h$.

To prove~\eqref{e2.1}, using~\eqref{fident} again we get for any $B \in L^1(0,r_0)$ and any $(\bv,\bw)$
\begin{align}\nonumber 
 &e = \int_0^{r_0} \biggl[ \left(\bv' + \oh\left(\frac{r}{R} - \bw'\right)^2 + B \right)^2 + \alpha_s h^{-2} |\bw|^2 + h^2 \left|\bw''-\frac12\right|^2\!\! +\! \W(\bv / r)\biggr] r \dr\! - \!F_h(\bv_h,\bw_h)
\\ \nonumber
&\ge
 \int_0^{r_0}\! \biggl[ \left(\!\bv' + \oh \left(\frac rR - \bw'\right)^2\!\! +\! B - \sigma_h \right)^2\!\!\! + \sigma_h \!\left(\bw' - \bw_h'\right)^2\!\! + \alpha_s h^{-2} (\bw - \bw_h)^2\!\! + h^2\! \left(\bw'' - \bw_h''\right)^2 \biggr]\! r \dr
\\ \nonumber
&\quad + \int_0^{r_0} 2\sigma_h(r) B(r) r \dr + \int_0^{r_0} \biggl( \int_{\bv_h(r)/r}^{\bv(r)/r} \W''(\xi) (\bv(r)/r-\xi) \ud \xi \biggr) \dr,
\end{align}
where the last term is non-negative by the convexity of $\W$. Using interpolation we get that $\int_0^{r_0} (\bw' - \bw_h')^2 r \dr \lesssim e$, which then implies $\int_0^{r_0} (\bv' - \bv_h')_+ r \dr \le \int_0^{r_0} |\bv' - \bv_h' + B| r \dr \lesssim e + e^{1/2} \lesssim e^{1/2}$. Arguing as above we can show that $\bv_h(r) >  -2\alpha_s^{1/2}$ for small (but not too small) values of $r$.
Combining this with the strict positivity of $\W''$, we deduce for such $r$ the smallness (in terms of $e$) of the $L^2$ norm of $\bv - \bv_h$. This can be upgraded using a Sobolev inequality to get 
$\int_0^{r_0} (\bv - \bv_h)_+^2 r \dr \lesssim e$. The proof of Proposition~\ref{prop} is now complete. 

\end{proof}

{
\section{The upper bound}\label{sec:ub}

\providecommand{\N}{\mathbb{N}}
\providecommand{\uo}{u_{\textrm{osc}}}
\providecommand{\wo}{w_{\textrm{osc}}}
\providecommand{\kopt}{k_{\textrm{opt}}}

To prove the upper bound, it is sufficient to define for each $h \in (0,\h)$ a deformation $(u_h,w_h)$ with the property $E_h(u_h,w_h) \le F_h(\bv_h,\bw_h) + \kappa(1/h) h$. Indeed, combining this with~\eqref{fhfh} gives the desired upper bound $E_h(u_h,w_h) \le F_0(\bv_0) + \kappa(1/h) h$ (with a possibly different $c_1$ in the~definition of $\kappa$). 
Since we do not attempt to get the optimal $\kappa$, it is enough to do this only for some  discrete set of $h$, which has $0$ as its limit point and has the property that neighboring $h$ differ at most by a factor of $2$. In practice, we will eventually restrict our attention to values of $h$ such that $h^{\delta - 1/2}$ is an integer, where the value of $\delta$ is close to $0$. 
To simplify the notation, we drop the~subscript $h$ from the deformation to be constructed, writing $(u,w)$ rather than $(u_h,w_h)$.

Our construction is guided by the proof of the lower bound. The basic idea is to modify the minimizer 
$(\bv_h,\bw_h)$ of $F_h$ by wrinkling where necessary, and estimate the increase in the energy due to wrinkling.
Using~\eqref{e2.28}, the amount of arclength we need to waste at each circle (together with the optimal length scale of 
wrinkling) can be read off from $(\bv_h,\bw_h)$. Nevertheless, learning from the proof of the lower bound we 
anticipate that we should not really use wrinkling with the optimal period, since it will be costly to change 
this period too often -- $B(r)$ would then be too large (in fact of order $O(1)$).
Anyway, as we explained in Section~\ref{sec:model}, it is not necessary for the wavenumber $k$ 
to take the ``optimal'' value $\alpha_s^{1/4} r/h$ at radius $r$. Rather, what matters is that 
$|\frac{hk}{r} - \alpha_s^{1/2} \frac{r}{kh}|$ be at most of order $h^{1/2}$. We shall achieve this 
by making the wrinkling modes appear/disappear over length scales of order $h^{1/2}$. 
It takes some time to motivate the~construction. We shall explain the key ideas in two passes.
\bigskip 

\noindent {\sc First pass.}
Our first pass is unsuccessful, but still informative. As noted above we propose to use just choices
of $k$ that are integer multiples of $h^{-1/2}$, changing from one $k$ to the next on a length scale of order $h^{1/2}$. Since the amplitudes of the modes change over scale $h^{1/2}$, one finds after some 
calculation that $B(r)$ is of order $h$ (which is OK). But what does it mean to ``change from 
one $k$ to the next?'' The obvious (though ultimately unsuccessful) idea is to 
use ``building blocks,''  as done for example 
in~\cite{benny-fissioning,BellaKohnDrapes,romandrapes}. The building
block between two radii (say, $r_1$ and $r_2$) would have single-mode wrinkling at the two extremes
($r=r_1$ and $r=r_2$) and a suitable interpolation in the middle. A standard approach to this
interpolation would be to take 
$w(r,\theta) = \bw_h(r) + f(r) k_1^{-1} \sin(k_1\theta) + g(r) k_2^{-1} \sin(k_2 \theta)$, where 
$k_1,k_2$ are the optimal choices at $r_1$ and $r_2$ respectively, and $f=1,g=0$ near $r_1$ while 
$f=0,g=1$ near $r_2$. Recall that $k_1, k_2 \sim h^{-1}$, while
by our choice of $r_1$ and $r_2$, $|k_1 - k_2| \sim h^{-1/2}$. To make the second term in 
the ``error'' $R_h$ (see~\eqref{Rh}) negligible one needs to choose $-\partial_\theta u_\theta$ to be 
approximately the deviation of $(\partial_\theta w)^2/2r$ from its average. 
Computing $(\partial_\theta w)^2 = f^2(r) \sin^2(k_1\theta) + g^2(r) \sin^2(k_2\theta) + 
2f(r)g(r) \sin(k_1\theta)\sin(k_2\theta)$, we see that the first two terms are of order 
$1$ with a period of order $h$ (since $k_1,k_2 \sim h^{-1}$), and therefore their contribution 
to $u_\theta$ (i.e. after integrating once in $\theta$) will be of order $h$. 
However the remaining term is problematic. Indeed, it can be written as
$f(r)g(r) [-\cos((k_1+k_2)\theta) + \cos((k_1-k_2)\theta))]$. The $(k_1 + k_2)$ term is not 
harmful by the same argument as above, however the $(k_1 - k_2)$ term is problematic. In fact 
it provides a term of order $h^{1/2}$ in the expression for $u_\theta$; this occurs because
$(k_1-k_2) \sim h^{-1/2}$ whereas $k_1$, $k_2$, and $k_1+k_2$ are all of order $h^{-1}$. 
Moreover, since $w$ changes with $r$ on length scale $h^{1/2}$, so should $u_\theta$. Thus
one expects $\partial_r u_\theta \sim 1$. This makes the third term in $R_h$ (the cross term) 
of order $1$ -- much too large. 

In summary: we need a better answer for what it should mean to ``change from one $k$ to the~next.'' 
\bigskip 

\noindent {\sc Second pass.} To get started we need some notation. 
We will consider an ansatz $(u,w)$ of the~form
$u_r(r,\theta) := \bv_h(r) + \uo(r,\theta)$ and $w(r,\theta) = \bw_h(r) + \wo(r,\theta)$,
where for each $r \in (0,r_0)$ we require $\int_0^{2\pi} \uo(r,\cdot) = \int_0^{2\pi} \wo(r,\cdot) = 0$. Then using~\eqref{e2.12} we get that
\begin{equation}\label{e3.1}
\begin{aligned}
 E_h(u,w)& - F_h(\bv_h,\bw_h) =
 \\
 &\int_0^{r_0} \biggl[ 2\sigma_h B(r) + B^2(r) + W_r\left(\frac{\bv_h(r)}{r}, w(r,\cdot) \right) - \W\left(\frac{\bv_h(r)}{r}\right) \biggr] r \dr + R_h(u,\xi),
\end{aligned}
\end{equation}
where recall that $B(r) = \fint_0^{2\pi} |\partial_r(\bar w - w)|^2 \dt$, $\sigma_h = \bv_h' + \frac12 (\frac rR - \bw_h')^2$, and $\xi(r,\theta) = w(r,\theta) - \frac{r^2}{2R}$. Using~\eqref{e2.28}, we know that $W_r - \W$ completely vanishes in the tensile region (i.e. where $\bv_h(r) / r \ge -2\alpha_s^{1/2}$ and $\wo = 0$) and will be small also in the rest of the domain (i.e. in the~wrinkled region) provided that both
$\sum_k a_k(r)^2 (k/r)^2$ is approximately equal to $- 2 (\bv_h(r)/r + 2\alpha_s^{1/2})$ and $(hk/r - \alpha_s^{1/2}r/(hk))^2$ is small in the support of $a_k$. The first condition 
says that the~wrinkling wastes the right amount of arclength, while the second condition says the wrinkling should be near
the optimal frequency $k = \alpha_s^{1/4}r/h$.

The key issue is how to avoid the problematic $\cos((k_1-k_2)\theta)$ term that appeared in $u_\theta$ in the first pass.
To explain the idea, we pretend for the moment that the~frequencies $k$ are allowed to be real-valued, not just integers.
Let us 
define $\wo(r,\theta) = A(r)r \int_\R h^{1/2} m[h^{1/2}(k - \alpha_s^{1/4} r/h)] \cos(k\theta)/k \ud k$, where $m$ is a 
non-negative smooth mask
(specifically: a nonnegative function supported in $(-1/2,1/2)$ such that $\int_\R m(t) \ud t = 1$). The function $A(r)$ 
modulates $w$ and should be chosen so that the wrinkled profile wastes the right amount of arclength.
For such $w$ we compute $(\partial_\theta w)^2 = hA^2(r) \int_\R \int_\R m[h^{1/2}(k - r/h)] m[h^{1/2}(l - r/h)] \sin(k\theta) \sin(l\theta) \ud k \ud l$. As before we want to compute the integral (in $\theta$) of this quantity, which has the form $\tfrac{1}{2} hA^2(r) \int_\R \int_\R m[h^{1/2}(k - r/h)] m[h^{1/2}(l - r/h)] (\tfrac{\cos((k+l)\theta)}{k+l} - \tfrac{\cos((k-l)\theta)}{k-l}) \ud k \ud l$. Focusing on the latter ``troublesome'' integral involving the term $\cos((k-l)\theta)/(k-l)$, we see that while its value is not small it {\it does not change} in $r$ (except for the dependence on $A$), which can be seen by the~change of variables $\hat k = k-r/h, \hat l = l-r/h$ in the double integral. Therefore, the contribution to $u_\theta$ from this part will be $r$-independent, and so $\partial_r u_\theta$ will not be too large. For discrete frequencies this problematic quantity will become $h$-periodic in $r$; the periodicity can be used to show that it is almost constant with very small derivative.

The argument just sketched \emph{almost} works. Unfortunately it doesn't quite work, since when the $r$-derivative of 
$u_\theta$ falls on the $A(r)$-term, one seems to need that $u_\theta$ itself is small -- which is unfortunately not true. 
To overcome this difficulty, the argument presented below includes a~further tweak -- it uses only frequencies that are
multiples of $h^{\delta - 1/2}$ for some $\delta > 0$. 
\bigskip 

In the rest of this section we use the preceding ideas to give an honest proof of the upper bound. 
To get started, we fix a small $\delta > 0$ and we require from now on that all the constants be independent of $\delta$. For $r \in (0,r_0)$ and $\theta \in [0,2\pi)$ we define $w(r,\theta) := \wo(r,\theta) + \bw_h(r,\theta)$ with
\begin{equation}\nonumber
 \wo(r,\theta) := A(r)r h^{\delta/2} \sum_{k>0} m\left[h^{\delta}k - \alpha_s^{1/4}rh^{-1/2}\right] \frac{\sqrt{2} \cos(kN\theta)}{kN},
\end{equation}
where $m$ should as before be a smooth non-negative mask, $A$ will be chosen later, and $N := h^{\delta-1/2}$ has without loss of generality an integer value.
Since in the following we will need estimates on derivatives of $m$, we make a particular choice $m(t) := \exp(-1/(1-4|t|^2))$ if $|t| < 1/2$ and $m=0$ elsewhere.
\providecommand{\wt}{\widetilde w_{\textrm{osc}}}

To estimate the excess energy, we now estimate term by term the right-hand side of~\eqref{e3.1}. Using definition and properties of $W_r$ and $\W$ (see~\eqref{Wenergy} and \eqref{e2.28}), we estimate difference of the third and fourth term in~\eqref{e3.1}:
\begin{align}
 W_r\biggl(\frac{\bv_h(r)}{r}&, w(r,\cdot) \biggr) - \W\left(\frac{\bv_h(r)}{r}\right) \nonumber
\\
&= \frac{1}{4} \left|\sigma(r) + 2(\frac{\bv_h(r)}{r}+2\alpha_s^{1/2})\right|^2 \label{eq5.1}
\\
&\quad + A^2(r) h^{\delta} \sum_k m^2\left[h^\delta k - \alpha_s^{1/4} rh^{-1/2}\right]\left(\frac{h^{1/2+\delta}k}{r} - \alpha_s^{1/2}\frac{r}{h^{1/2+\delta}k}\right)^2, \label{eq5.1+}
\end{align}
where $\sigma(r) = A^2(r) h^{\delta} \sum_{k>0} m^2[h^\delta k - \alpha_s^{1/4} rh^{-1/2}]$. To estimate~\eqref{eq5.1+} we observe that the~function $f(t) = h^{1/2+\delta}t/r - \alpha_s^{1/2} r/(h^{1/2+\delta}t)$ vanishes for $t=\alpha_s^{1/4} r/h^{1/2+\delta}$ and $|f'| \lesssim h^{1/2+\delta}/r$ nearby, and by the support condition for $m$ we see that~\eqref{eq5.1+} is bounded from above by $A^2(r)h^{\delta} h^{-\delta} [(h^{1/2+\delta}/r)h^{-\delta}]^2 = A^2(r)h/r^2$.  Since in the following we will choose $A$ such that $|A| \lesssim 1$ and $A$ will be supported away from the origin, we see that we will have $A^2(r)h/r^2 \lesssim h$.
To simplify the notation, from now on we will assume that $\alpha_s = 1$ (in the general case all the~subsequent constants might depend also on $\alpha_s$).

In order to make~\eqref{eq5.1} small, we would like to choose value of $A(r)$ such that $\widetilde \sigma(r) := A^2(r) h^{\delta} \int_\R m^2[h^\delta k - \alpha_s^{1/4} rh^{-1/2}]
\ud k = A^2(r) \int_\R m^2(\hat k) \ud \hat k = -2(\frac{\bv_h(r)}{r}+2\alpha_s^{1/2})$ if the right-hand side is positive (wrinkled region), and $A(r)=0$ if it is non-negative (non-wrinkled region).
We have introduced $\widetilde \sigma$ as a proxy (a less oscillating approximation) for $\sigma$, since we do not want $A$ to oscillate on scale $h^{\delta}$, which would be inevitable if we  defined $A$ using $\sigma(r) = -2(\frac{\bv_h(r)}{r}+2\alpha_s^{1/2})$. The advantage of using $\widetilde\sigma$ instead of $\sigma$ to define value of $A$ is that in this case $A^2$ is as smooth as $\bv_h$. 
In addition, since we will also need control on derivatives of $A$ (and not only $A^2$), we cut off $A$ on scale $h$ near the transition between the flat and wrinkled region (where derivatives of $A$ would be singular):

\begin{equation}\nonumber
 A(r) := \begin{cases} 0 & r \in (0,r_w), \\
\eta\left(\frac{r-r_w}{h}\right) \left(-2\left(\frac{\bv_h(r)}{r} + 2\alpha_s^{1/2}\right)\right)^{1/2} \left(\int_\R m^2\right)^{-1/2} & r \in (r_w,r_0),
        \end{cases}
\end{equation}
with $\eta$ being a smooth cutoff for $(2,\infty)$ in $(1,\infty)$ (i.e. $\eta(t)=1$ if $t>2$ and $\eta(t)=0$ if $t<1$). While without the cutoff $\eta$ the derivative $A'(r)$ would blow-up like $(r-r_w)^{1/2}$, due to this cutoff we see that $|A'| \lesssim h^{-1/2}$, and similarly $|A''| \lesssim h^{-3/2}$.

To estimate~\eqref{eq5.1}, we will use the following simple observation: For any smooth compactly supported function $f$ 
and any $n \in \N$ there exists $C$, which depends on the support of $f$, such that for any $t \in (0,1)$ and any shift $\zeta \in \R$ we have
\begin{equation}\label{eq6.1}
\biggl|t\sum_{k\in \mathbb{Z}} f(tk+\zeta) - \int_\R f \biggr| \le Ct^n\left\|f^{(n)}\right\|_{\infty} .
\end{equation}
Indeed, for $n=0$ this holds more generally for $m$-th derivative of $f$, since $\int_\R f^{(m)}=0$ implies 
\begin{align*}
\biggl|t\sum_{k\in \mathbb{Z}} f^{(m)}(tk+\zeta) - \int_\R f^{(m)} \biggl| &= t\biggl|\sum_{k\in \mathbb{Z}} f^{(m)}(tk+\zeta)\biggl| \le t \#\left\{ k : f^{(m)}(tk+\zeta) \neq 0 \right\} \left\| f^{(m)} \right\|_{\infty}
\\
&\le C\left\| f^{(m)} \right\|_{\infty},
\end{align*}
where $\| g \|_\infty := \sup |g|$. Moreover, for any $m \ge 1$ we have 
\vspace{-2pt}
\begin{align*}
\biggl|\sum_{k\in \mathbb{Z}} f^{(m)}(tk+\zeta)\biggl| &= \biggl|\sum_{k\in \mathbb{Z}} f^{(m)}(tk+\zeta) - t^{-1}\int_\R f^{(m)}\biggl| = \biggl|\sum_{k\in \mathbb{Z}} f^{(m)}(tk+\zeta) - f^{(m)}(tk+\zeta')\biggl| 
\\
&\le t \biggl|\sum_{k\in \mathbb{Z}} f^{(m+1)}(tk+\zeta'')\biggl|,
\end{align*}
and so by induction we get~\eqref{eq6.1} for any derivative of $f$. In addition, since the previous chain of equalities (except for the first one) holds also for $m=0$, we get~\eqref{eq6.1} also for $f$ itself, which finishes the proof of~\eqref{eq6.1}.

This observation, applied to $f=m^2$ with $n \ge 1/\delta$, yields after a change of variables the~estimate 
\begin{align*}
|\sigma(r) - \widetilde \sigma(r)| &= A^2 |h^{\delta} \sum_k m^2[h^\delta k] - \int_\R m^2[h^\delta k] \ud k| 
\\
&\le C A^2 h^{\delta} (h^{\delta})^{n-1} \| (m^2)^{(n)} \|_{\infty}\le C\| (m^2)^{(n)} \|_{\infty}h \le (C_m/\delta)^{(C_m'/\delta)} h,  
\end{align*}
where in the last step we used an estimate for $\| (m^2)^{(n)} \|_{\infty}$. Using the triangle inequality we have $|\sigma(r) + 2(\frac{\bv_h(r)}{r}+2\alpha_s^{1/2})|^2 \lesssim |\widetilde \sigma(r) + 2(\frac{\bv_h(r)}{r} + 2\alpha_s^{1/2})| + |\sigma(r) - \widetilde\sigma(r)|^2$. Since the first term on the right-hand side vanishes if $r-r_w \not \in (0,2h)$ and is $O(1)$ inside this interval, the previous estimate on $\sigma(r) - \widetilde \sigma(r)$ implies that~\eqref{eq5.1} is bounded by $(C_m/\delta)^{(C_m'/\delta)} h$.

Next we turn to $B(r) = \fint_0^{2\pi} |\partial_r \wo|^2 \dt = h^{\delta} \sum_k (kN)^{-2} [\partial_r(A(r)rm[h^\delta k-rh^{-1/2}])]^2$. From now on we will repeatedly use the fact that due to the support condition on $m$, we have $m[h^\delta k-rh^{-1/2}] \neq 0$ only for at most $h^{-\delta}$ values of $k$, and $kN \gtrsim (r/h)$. Due to this, and also since $|A'| \lesssim h^{-1/2}$, we see that $|B(r)| \lesssim h^\delta h^{-\delta} (r/h)^{-2} (h^{-1/2})^2 \lesssim h$, where the last $h^{-1/2}$ comes from derivative in $r$.
Hence in the wrinkled region $|B(r)| \lesssim h$ and $B=0$ otherwise, and so $\int_0^{r_0} (\sigma_h B(r) + B(r)^2) r \dr \lesssim h$.

We now estimate the remainder $R_h$, which consists of five terms:
\begin{equation}\nonumber
\begin{aligned}
 &\int_0^{r_0} \fint_0^{2\pi} \biggl| \partial_r (u_r-\bar u_r) + \frac{(\partial_r \xi)^2}{2} - \fint_0^{2\pi} \frac{(\partial_r \xi)^2}{2}\biggl|^2 \dt\ r \dr
 \\
 &+ \int_0^{r_0} \fint_0^{2\pi} \biggl|\frac{\partial_\theta u_\theta}{r} + \frac{u_r}{r} + \frac{(\partial_\theta w)^2}{2r^2} - \frac{\bar u_r}{r} - \fint_0^{2\pi} \frac{(\partial_\theta w)^2}{2r^2} \biggr|^2 \dt\ r \dr
\\
& +\! \int_0^{r_0}\! \biggl( \fint_0^{2\pi}
\!\! \frac 12 \biggl|\frac{1}{r} \partial_\theta u_r\! +\! r \partial_r \left(\frac{u_\theta}{r}\right) + \frac{1}{r} \partial_r \xi \partial_\theta \xi\biggl|^2\! + h^2 |\partial_{rr}(w - \bw)|^2\! +\! \frac{2h^2}{r^2} |\partial_{\theta r} \xi|^2 \dt\ \!\!\biggr) r \dr =: \sum_{i=1}^5 T_i.
\end{aligned}
\end{equation}
To do the estimates we first define $u_\theta$ and $u_r$:
\begin{align*}
 u_{\theta,+}(r,\theta) &:= A^2(r) r^2 h^{\delta} \sum_{k,l} m\left[h^\delta k-rh^{-1/2}\right]m\left[h^\delta l-rh^{-1/2}\right]\frac{\sin((k + l)N\theta)}{(k+l)N},
\\
 u_{\theta,-}(r,\theta) &:= A^2(r) r^2 h^{\delta} \sum_{k\neq l} m\left[h^\delta k-rh^{-1/2}\right]m\left[h^\delta l-rh^{-1/2}\right]\frac{\sin((k- l)N\theta)}{(k-l)N},
\\ u_\theta(r,\theta) &:= \frac{1}{2r} (u_{\theta,+} - u_{\theta,-}), \qquad\qquad\qquad u_r(r,\theta) := \frac{r}{R} \wo(r,\theta) + \bar u_r(r).
\end{align*}


To estimate $T_1$, we use definition of $\wo$, together with the support condition on $m$ and the fact $kN \gtrsim h^{-1}$, to see that $|\wo| \lesssim h^{1-\delta/2}$ and $|\partial_r \wo| \lesssim h^{1/2-\delta/2}$. Therefore we immediately see that $|\partial_r(u_r-\bar u_r)| \le \frac 1R |\wo| + \frac rR |\partial_r \wo| \lesssim h^{1/2-\delta/2}$. Using $\xi(r,\theta) = \wo(r,\theta) + \bw_h(r) - \frac{r^2}{2R}$ we get
$(\partial_r \xi)^2 - \fint_0^{2\pi} (\partial_r \xi)^2 = (\partial_r \wo)^2 - \fint_0^{2\pi} (\partial_r \wo)^2 + 2 \partial_r \wo \partial_r \bw_h - 2\frac rR \partial_r \wo$, which by previous estimates on $\wo$ and $\partial_r \wo$ combined with estimate on $\partial_r \bw_h$ (see~\eqref{e2.15}) implies $|T_1| \lesssim h^{1-\delta}$.
Using definition of $u_\theta$ we see that $\partial_\theta u_\theta + [(\partial_\theta w)^2 - \fint_0^{2\pi} (\partial_\theta w)^2]/2r = 0$, which together with the~fact $(u_r - \bar u_r)/r = \wo/R$ and estimate $|\wo|\lesssim h^{1-\delta/2}$ implies $|T_2| \lesssim h^{2-\delta}$.

To estimate $|T_3|$ we first focus on $|\partial_r(u_\theta/r)|$, by dealing with the $+$ and $-$ part separately. For the $+$ part we have
\begin{align*}
 \partial_r (&u_{\theta,+}/r^2) =\\
 &h^{\delta} \biggl( \partial_r (A^2(r)) \sum_{k,l} m[\ldots]m[\ldots] \frac{\sin((k+l)N\theta)}{(k+l)N} +
 A^2(r) \sum_{k,l} \partial_r(m[\ldots]m[\ldots])\frac{\sin((k+l)N\theta)}{(k+l)N}\biggr),
\end{align*}
where here and below we use an abbreviation $m[\ldots]m[\ldots] = m[h^\delta k-rh^{-1/2}]m[h^\delta l-rh^{-1/2}]$.  
Since the summation is performed over $h^{-2\delta}$ pairs of $k,l$, $|\partial_r (A^2)| \lesssim 1$, and $(k+l)N \gtrsim h^{-1}$, the~first half $\lesssim h^\delta h^{-2\delta} h \le h^{1-\delta}$. The second half is different since the derivative falls on $m$, which introduces  additional factor $h^{-1/2}$, leading to the estimate $\lesssim h^{1/2-\delta}$. Altogether we see that $|\partial_r (u_{\theta,+}/r^2)| \lesssim h^{1/2-\delta}$.
For the $-$ part we have
\begin{align*}
 \partial_r (&u_{\theta,-}/r^2) =\\
 &h^{\delta} \biggl( \partial_r (A^2(r)) \sum_{k\neq l} m[\ldots]m[\ldots] \frac{\sin((k-l)N\theta)}{(k-l)N} +
 A^2(r) \sum_{k\neq l} \partial_r(m[\ldots]m[\ldots])\frac{\sin((k-l)N\theta)}{(k-l)N}\biggr). 
\end{align*}
Using arguments as above and $N = h^{-(1/2-\delta)}$ the first half $\lesssim h^\delta h^{-2\delta} N^{-1} = h^{1/2-2\delta}$. This idea would not be enough for the second half, since the additional factor $h^{-1/2}$ would ruin the estimate. Instead, for fixed $\theta$ let us consider $f(r) := h^{3\delta}\sum_{k\neq l} m[h^\delta k - rh^{-1/2}]m[h^\delta l - rh^{-1/2}]\frac{\sin((k-l)N\theta)}{(k-l)N}$ and observe that by a change of variables $f(r)$ is $h^{1/2+\delta}$-periodic.
Moreover, we see that for any $n \in \N$ we have $|f^{(n)}| \le (C_m n)^{(C_m'n)}h^{(1-n)/2}$, where we used estimates on $m^{(n)}$ and the fact that each derivative introduces a factor $h^{-1/2}$. Let us now fix $n \ge 1$, and observe that $f^{(n)}$ has to have mean zero (due to periodicity of $f^{(n-1)}$), in particular there exists $t \in (0,h^{1/2+\delta})$ such that  $f^{(n)}(t)=0$.
Using the bound on $f^{(n+1)}$ and first-order Taylor expansion, we get for any $s \in (0,h^{1/2+\delta})$ that $|f^{(n)}(s)| \le \|f^{(n+1)}\|_{\infty} h^{1/2+\delta} \le (C_m n)^{(C_m'n)}h^{(-n)/2} h^{1/2+\delta} \le (C_m n)^{(C_m'n)} h^{\delta+(1-n)/2}$. We can now iterate such an estimate, and since in each step we improve the exponent by $\delta$, after $1/(2\delta)$ steps we get that $|f'|\le (C_m n)^{(C_m'n)}h^{1/2}$ with $n \sim 1/\delta$, which then leads to an estimate $|\partial_r (u_{\theta,-}/r^2)|\le (C_m/\delta)^{(C_m'/\delta)} h^{1/2-2\delta}$. The estimate $|\partial_r (u_\theta/r)|\le (C_m/\delta)^{(C_m'/\delta)} h^{1/2-2\delta}$ 
follows immediately.

Using that $\xi(r,\theta) = \wo(r,\theta) + \bw_h(r) - r^2/2R$ and the definition of $u_r$, we see that
\begin{equation*}
 \partial_\theta u_r + \partial_r \xi \partial_\theta \xi = (\partial_r \wo + \partial_r \bw_h)\partial_\theta \wo.
\end{equation*}
Since $|\partial_\theta \wo| \lesssim h^{-\delta/2}$, which we get by direct computation, this in combination with the estimate on $\partial_r \wo$ and $\partial_r \bw_h$ yields $|\partial_\theta u_r + \partial_r \xi \partial_\theta \xi|\lesssim h^{1/2-\delta}$, and so $|T_3| \lesssim h^{1-4\delta}$.

To deal with $T_4$, we observe that $\partial_{rr}(w-\bar w) = \partial_{rr} \wo$, which then using bounds on $A'$ and $A''$ can be estimated by $C h^{-(1+\delta)/2}$, which together with the additional $h^2$ factor gives $|T_4| \lesssim h^{1-\delta}$.

Finally, we see that $|\partial_{r\theta} \wo| \lesssim h^{-1/2-\delta/2}$, and $|T_5| \lesssim h^{1-\delta}$ immediately follows. Altogether we have shown that the excess energy is bounded by $(C_m/\delta)^{(C_m'/\delta)} h^{1-4\delta}$. By choosing $\delta := (-\log h)^{-1/2}$ we obtain an estimate $\epsilon \le \exp\big(C(\log(1/h))^{1/2} \log (\log(1/h)) \big) h$, and the proof of the upper bound is complete.

With this choice of $\delta$, we also observe that a simple estimate $|\partial_\theta \wo| \le C h^{-\delta/2}$ turns into $|\partial_\theta \wo| \le C \exp( (\log 1/h)^{1/2} )$. Hence, though the slopes are not uniformly bounded in $h$, they explode with a rate slower than any power of $h$.



\section*{Acknowledgement}

Support is gratefully acknowledged from NSF grants DMS-1311833 and
OISE-0967140 (RVK), and from DFG grant BE 5922/1.1 (PB).

This project was initiated as a result of discussions with Benny Davidovitch and Gregory Grason, which we gratefully acknowledge. Those discussions took place at the Kavli Institute for Theoretical Physics, during a spring 2016 program on Geometry, Elasticity, Fluctuations, and 
Order in 2D Soft Matter (supported by NSF through grant PHY11-25915). We also gratefully acknowledge input from the anonymous referees, whose comments led to substantial improvements in the presentation of our results. PB was a postdoc at the Max Planck Institute for Mathematics in the Sciences (Leipzig) when this work was begun.

\bibliographystyle{amsplain}
\bibliography{bella-with-additions}

\end{document}